\documentclass{amsart}
\usepackage{amssymb}
\usepackage[all]{xy}

\newcommand\boxb[1]{\square_b}

\newcommand\paperbody%
        {}

\newtheorem{lemma}{Lemma}
\newtheorem{proposition}{Proposition}

\newtheorem{theorem}{Theorem}
\newtheorem{non-theorem}{Non-Theorem}

\theoremstyle{remark}

\newcommand\cFTs{{}^{\Phi}\overline{T}\kern-1pt{}^*}

\newcommand\ie{i.e.\ }

\newcommand\bbR{\mathbb R}

\newcommand\CI{{\mathcal{C}}^{\infty}}

\newcommand\CmI{{\mathcal{C}}^{-\infty}}

\newcommand\cFNs{{}^{\Phi}\overline N\kern-1pt{}^*}

\newcommand\dCI{\dot{\mathcal{C}}^{\infty}}

\newcommand\pa{\partial}

\newcommand\supp{\operatorname{supp}}

\newcommand\Mif{\text{ if }}

\newcommand\Mwhere{\text{ where }}
\newcommand\Mwith{\text{ with }}

\begin{document}
\title[Distributions and the de Rham theorem]
{A remark on distributions and the de Rham theorem}

%lasteqno dR@ 11
\author{Richard B. Melrose}
\address{Department of Mathematics, Massachusetts Institute of Technology}
\email{rbm@math.mit.edu}
\dedicatory{\today}

\begin{abstract} We show that the de Rham theorem, interpreted as the
isomorphism between distributional de Rham cohomology and simplicial
homology in the dual dimension for a simplicial decomposition of a compact
oriented manifold, is a straightforward consequence of elementary properties
of currents. The explicit construction of this isomorphism extends to other
cases, such as relative and absolute cohomology spaces of manifolds with
corners.
\end{abstract}
\maketitle

%\section{Introduction \label{S.Introduction}}

The de Rham theorem (\cite{MR49:11552,MR85m:58005}) is generally
interpreted as the isomorphism, for a compact oriented manifold $X,$
between the cohomology of the de Rham complex of smooth forms 
\begin{equation}
0\longrightarrow \CI(X)\longrightarrow \CI(X;\Lambda ^1)\longrightarrow
\cdots\longrightarrow \CI(X;\Lambda ^n)\longrightarrow 0,
\label{dR.6}\end{equation}
where $\dim X=n,$ and the simplicial, or more usually the \v Cech,
cohomology of $X.$ This isomorphism is constructed using a double complex;
for proofs of various stripes see \cite{MR54:1273}, \cite{MR39:6343} or
\cite{MR84k:58001}.

The distributional de Rham cohomology, the cohomology of the complex
\eqref{dR.6} with distributional coefficients (currents in the terminology
of de Rham),
\begin{equation}
0\longrightarrow \CmI(X)\longrightarrow\CmI(X;\Lambda
^1)\longrightarrow\cdots\longrightarrow \CmI(X;\Lambda ^n)\longrightarrow0,
\label{dR.9}\end{equation}
is naturally Poincar\'e dual to the smooth de Rham cohomology using the
integration map $(\alpha ,\beta )\longmapsto\int_X \alpha \wedge \beta.$

Here we show that there is a relatively simple retraction argument which
shows that the homology of \eqref{dR.9} is isomorphic to the simplicial
homology, for any simplicial decomposition, in the dual dimension. The map
from simplicial to distributional de Rham cohomology takes a simplex to its
Poincar\'e dual (see for example \cite{Bott-Tu1}). There are many possible
variants of the proof below and in particular it is likely to apply to
intersection type homology theories on compact manifolds with corners.

I would like to thank Yi Lin for pointing out an error in the proof of
Lemma~\ref{11.11.1999.1} in an earlier version.

\section{Distributions and currents}

We use some results from distribution theory which are well known. These
are mainly to the effect that a simplex is `regular' as a support of
distributions.

\begin{lemma}\label{dR.4} Any extendible distribution on the interior of
an $n$-simplex $S\subset\bbR^n,$ \ie an element of the dual of
$\dCI(S;\Omega)=\{u\in\CI(\bbR^n);\supp(u)\subset S\},$ is the restriction
of a distribution on $\bbR^n$ with support in $S.$
\end{lemma}
\noindent Here $\Omega$ is the density bundle.

\begin{lemma}\label{dR.10} Any current with support in a plane
$\bbR^k_y\times\{0\}_z\subset\bbR^n$ is of the form 
\begin{equation}
\sum\limits_{\alpha,I}\delta^{(\alpha)}(z)dz^Iu_{\alpha ,I}(y)
\label{dR.11}\end{equation}
where the $u_{\alpha ,I}$ are currents on $\bbR^k$ and the $\delta
^{(\alpha )}=\pa^\alpha \delta (z)/\pa z^\alpha$ are derivatives of the
Dirac delta function.
\end{lemma}

\begin{proposition}\label{11.11.1999.7} If $X$ is a manifold with a
simplicial decomposition then any distribution or current with support on
the $p$-skeleton is the sum of distributions with supports on the individual
$p$-simplexes.
\end{proposition}

\section{Poincar\'e Lemmas}

\begin{lemma}\label{10.11.1999.6} In $\bbR^n$ the complex of distributional
forms, \ie currents, with support at the origin has homology which is
one-dimensional and is in dimension $n.$
\end{lemma} 

\begin{proof} We must show that a closed $k$-current of this type is
always of the form $dv$ for a $k-1$ current supported at the origin, unless
$k=n$ in which case
\begin{equation}
u=c\delta(x)dx_1\wedge\cdots\wedge dx_n+dv.
\label{10.11.1999.8}\end{equation}
The key to this is simply the representation \eqref{dR.11} in this case,
decomposing currents supported at the origin as finite sums
\begin{equation}
u=\sum\limits_{\alpha ,I}c_{\alpha ,I}\delta ^{(\alpha)}(x)dx^I,
\label{10.11.1999.7}\end{equation}
where $\delta ^{(\alpha)}(x)=\pa_x^\alpha \delta (x).$ Each of these terms
is homogeneous of degree $-\alpha_i$ or $1-\alpha_i$ under the homotheity
$R_i^t$ where $R_i^t(x)= (x_1,\cdots,x_{i-1},tx_i,x_{i+1},\cdots,x_n).$
Since $d$ itself is invariant under these transformations it follows that
if $u$ is closed, so are each of the terms of fixed homogeneity in each
variable. Consider the identity for currents
\begin{equation}
t\frac{d}{dt}(R_i^t)^*u=(R_i^t)^*(d\mathcal{L}_i+\mathcal{L}_id)u
\label{dR.1}\end{equation}
where $\mathcal{L}_i$ is contraction with the radial vector field,
$x_i\pa_{x_i}$ in the $i$th coordinate. Then if $u$ is closed and homogeneous
of degree $a_i$ it follows that $\frac{d}{dt}(R _i^t)^*u|_{t=1}=a_iu=dv,$
$v=\mathcal{L}_iu.$ Thus all closed currents of non-zero multi-homogeneity
are exact. The only currents which are homogeneous of degree zero are the
multiples of $\delta(x)dx_1\wedge\cdots\wedge dx_n$ so we have proved
\eqref{10.11.1999.8}.

The lemma now follows from the fact that these forms are not themselves
exact. This again uses the same type of homogeneity argument. If
$\delta(x)dx_1\wedge\cdots\wedge dx_n=dv$ with $v$ supported at the origin,
then $v$ may be replaced by its homogeneous part of degree $0.$ Since there
are no currents of form degree $n-1$ which are homogeneous of degree $0$ it
follows that no such $v$ can exist.
\end{proof}

Next we compute the extendible distributional de Rham cohomology of the
interior of the standard $n$-simplex in $\bbR^n.$ This is also a form of
the Poincar\'e lemma.

\begin{lemma}\label{11.11.1999.1} If $u$ is a closed extendible $k$-current
on the interior of $S_n=\{x\in\bbR^n;0\le x_i\le1,0\le
x_1+\cdots+x_n\le1\}$ then $u=dv$ with $v$ an extendible $(k-1)$-current
unless $k=0$ in which case $u$ is a constant.
\end{lemma}

\begin{proof} If $u$ is a $0$-current, \ie a distribution, then $du=0$
implies that $u$ is constant. Thus we may assume that $k>0.$ 

We proceed by induction over the condition that there exists a
current $v_j$ such that $\mathcal{L}_i(u-dv_j)=0$ for all $i\le j.$ For the
first step we may write 
\begin{equation}
u=u'+ dx_1\wedge u''
\label{dR.2}\end{equation}
where $\mathcal{L}_1u'=\mathcal{L}_1u''=0$ are respectively a $k$ and a
$k-1$ current. Now, $u''$ may be considered as an element of a finite
tensor product of extendible distributional `functions' on $S_n$ with the
vector space of forms in the variables $x_j,$ $j>1.$ As such it can be
integrated in $x_1.$ That is, there exists an extendible form $v_1$ on
$S_n$ which satisfies $\mathcal{L}_1v_1=0$ and $\frac{\pa v_1}{\pa
  x_1}=u''.$ To construct $v_1,$ simply extend $u''$ to a compactly
supported distribution and then integrate, say from $x_1<<0,$ (which is
always possible) and then restrict this new distribution back to $S_n.$ 
It follows that $u_1=u-dv_1$ satisfies $\mathcal{L}_1(u_1)=0$ since
$$
\mathcal{L}_1(u_1)=u''-\mathcal{L}_1(u_1)(dx_1\wedge\frac{\pa v_1}{\pa x_1}+d'v_1)=0.
$$
Now we may proceed by induction since $du_1=0$ and
$\mathcal{L}_1u_1=0$ implies that $u_1$ is completely independent of $x_1,$
so subsequent steps are the same with fewer variables. When $j>n-k$ it
follows that $u$ is exact.
\end{proof}

Lemma~\ref{10.11.1999.6} is actually the zero dimensional case, and
Lemma~\ref{11.11.1999.1} essentially the $n$-dimensional case of the following
proposition in which we consider the standard $p$-simplex in $\bbR^n:$  
\begin{equation}
S_p=\{x\in\bbR^n;x_1=\cdots=x_{n-p}=0,\ x_j\ge0,\ j>n-p,\
x_{n-p+1}+\cdots+x_n\le1\}.
\label{10.11.1999.10}\end{equation}
The basic current we associate with $S_p$ is 
\begin{equation}
D(S_p)=\chi(S_p)\delta(x_1)\cdots\delta (x_{n-p})dx_1\wedge\cdots dx_{n-p}.
\label{10.11.1999.11}\end{equation}
Here $\chi (S_p)$ is the characteristic function of $S_p$ in the variables
$x_j,$ $j>n-p.$

\begin{proposition}\label{10.11.1999.12} If $u$ is a $k$-current on
$\bbR^n$ with support contained in $S_p$ and $du=0$ in $\bbR^n\setminus\pa
S_p$ then there is a $(k-1)$-current $v$ with support in $S_p$ such that 
\begin{equation}
u=\begin{cases}dv+u'&\Mif k\not= n-p,\\ dv+u'+cD(S_p)&\Mif k=n-p\end{cases}
\Mwith\supp(u')\subset\pa S_p.
\label{10.11.1999.13}\end{equation}
\end{proposition}

\begin{proof} Let us write the first $n-p$ variables as $y$ and the second
$p$ variables as $z.$ The decomposition analogous to \eqref{10.11.1999.7} for a
closed current in this case is 
\begin{equation}
u=\sum\limits_{\alpha ,I}\delta ^{(\alpha)}(y)dy^I\wedge u_{\alpha ,I}(z)
\label{10.11.1999.14}\end{equation}
where now the $u_{\alpha ,I}$ are $(k-|I|)$-currents on $\bbR^{p}$ with
support in $S_p\subset\bbR^p.$

The homogeneity argument of Lemma~\ref{10.11.1999.6} may now be
followed. Thus, $u$ may be decomposed into its multi-homogeneous parts
under separate scaling in each of the variables in $y;$ since $d$ again
preserves such homogeneity, each term is closed if $u$ is. As before, the
terms of non-zero homogeneity, in any of the $y$ variables, is exact near the
interior of $S_p.$ Thus for some $v'$ with support in $S_p$ 
$u'=u-dv$ has the same closedness property and is of the form
\begin{equation}
u'=dy^1\wedge\cdots dy^{n-p} u'' 
\label{dR.3}\end{equation}
where $u''$ is a $k-n+p$ form on $S_p\subset\bbR^p;$ in particular if
$k<n-p$ then $u-dv$ has support in $\pa S_p.$

It follows from \eqref{dR.3} that $u''$ is closed in the interior of $S_p$
as a $k-n+p$ form on $\bbR^p.$ Thus Lemma~\ref{11.11.1999.1} may be
applied. The extendible current constructed there, so that (unless
$k=n-p)$ $u''=dv''$ in the interior of $S_p$ may be extended to a current
$w$ with support in $S_p$ such that $u-dw$ has support in $\pa S_p.$ This
yields the desired result.
\end{proof}

\section{De Rham theorem}

Observe that the current $D(S_p)$ associated with the standard $p$-simplex
is invariant under oriented diffeomorphism of a neighbourhood of it in
$\bbR^n.$ Thus it is well defined for any oriented simplex in an oriented
manifold. In fact only the relative orientation, \ie orientation of the
normal bundle, is important.

\begin{theorem}\label{11.11.1999.2} Let $X$ be an oriented compact manifold,
without boundary, with a given simplicial decomposition, with (oriented)
simplexes labelled $S(j)=S_p(j)$ where $p$ is the dimension, then the chain
map
\begin{equation}
E:\sum\limits_{j}c_jS_p(j)\longmapsto
\sum\limits_{j}c_jD(S_p(j))\in\CmI(X;\Lambda ^{n-p})
\label{11.11.1999.3}\end{equation}
is a homology equivalence giving an isomorphism between the simplicial
$p$-homology of $X$ and its distributional $n-p$ de Rham cohomology.
\end{theorem}

\begin{proof} By direct computation, for the standard $p$-simplex, 
\begin{equation}
dD(S_p)=\sum\limits_{r}D(S_{p-1}(r))
\label{11.11.1999.4}\end{equation}
where $S_{p-1}(r)$ are the bounding $(p-1)$-simplexes with their induced
orientations. Thus the map does give a chain map: 
\begin{equation}
d(E(c))=E(\delta (c))
\label{11.11.1999.5}\end{equation}
where $\delta$ is the standard differential of simplicial homology.

To prove the theorem it suffices to show that the distributional de
Rham complex can be retracted onto the simplicial subcomplex. That is, 
\begin{equation}
\begin{gathered}
u\in\CmI(X;\Lambda ^k),\ du=0\Longrightarrow u=dv+E(c),\ \\
dE(c)=0\Longrightarrow \delta (c)=0.
\end{gathered}
\label{11.11.1999.6}
\end{equation}
In fact the second of these is clear, since $E$ is injective. We also need
the corresponding statements for exact forms. That is 
\begin{equation}
E(c)=dv,\ v\in\CmI(X;\Lambda ^*)\Longrightarrow c=\delta c'.
\label{11.11.1999.14}\end{equation}

Thus suppose $u$ is a closed $k$-current. Let $K_j$ denote the $j$ skeleton
of the simplicial decomposition, \ie the union of the
$j$-simplexes. Proceeding step by step we first decompose $u$ as a sum of
$k$-currents supported on each of the $n$-simplexes. Each of the terms is
closed in the interior of each simplex, so
Proposition~\ref{10.11.1999.12} may be applied to give a decomposition
\begin{equation}
u=\begin{cases}\sum\limits_{j}c_jD(S_n(j))&\Mif k=0\\
dv+u_1&\Mif k>0,
\end{cases}\Mwhere \supp(u_1)\subset K_{n-1}.
\label{11.11.1999.8}\end{equation}
We can ignore the case $k=0.$ Now it follows that $u_1$ is closed. 
Applying Proposition~\ref{11.11.1999.7} to decompose $u_1$ as a sum over
the $n-1$ skeleton and applying Proposition~\ref{10.11.1999.12} to each
part, gives a new decomposition and we may continue until we reach the $n-k$
skeleton. Thus we arrive at 
\begin{equation}
u=\sum\limits_{j}c_jD(S_{n-k}(j))+dv+u_{k+1},\ \supp(u_{k+1})\subset K_{n-k-1}.
\label{11.11.1999.9}\end{equation}
Let $c=\sum\limits_{j}c_jS_{n-k}(j)$ be the corresponding simplicial
chain, so the first term in \eqref{11.11.1999.9} is $E(c).$ Now
$dE(c)=E(\delta c).$ Thus, near the interior of any $n-k-1$ simplex,
$S_{n-k-1}(r),$ $du_{k+1}=-E(\delta c)=c'_rD(S_{n-k-1}(r)).$ However,
$D(S_{n-k-1}(r))$ is not in the range of $d$ on currents supported on the
corresponding simplex. Thus $c'_{r}=0$ for all $r$ which just gives $\delta
c=0.$ Thus 
\begin{equation}
u=E(c)+dv+u_{k+1},\ \supp(u_{k+1})\subset K_{n-k-1},\ \delta c=0,\ du_{k+1}=0.
\label{11.11.1999.10}\end{equation}
Now we can proceed successively, as before, and conclude that $u_{k+1}=dw$
with $w$ supported on the $n-k-1$ skeleton.

The arguments needed for \eqref{11.11.1999.14} are similar. Thus, it
follows from $E(c)=dv$ that $v$ is closed in the complement of the support
of $c.$ Assuming that $c$ is an $n-p$ chain, this means that $dv=0$ off the
$n-p$ skeleton. The argument above shows that $v=E(c')+dv'+w$ where $w$ has
support on the $n-p$ skeleton. Since $dE(c')=E(\delta c')$ we conclude that
$E(c-\delta c')=dw,$ with $w$ supported on the $n-p$ skeleton. As already
noted, this implies that $c=\delta c'.$

This proves the de Rham theorem.
\end{proof}

\providecommand{\bysame}{\leavevmode\hbox to3em{\hrulefill}\thinspace}

\end{document}